\documentclass{amsart}

\newtheorem{thm}{Theorem}[section]

\newtheorem{lemma}[thm]{Lemma}
\theoremstyle{definition}
\newtheorem{defn}[thm]{Definition}
\theoremstyle{remark}
\newtheorem{rem}[thm]{Remark}
\numberwithin{equation}{section}
\newcommand{\R}{\mathbb R}
\newcommand{\N}{\mathbb N}

\newcommand{\be}{\begin{equation}}
\newcommand{\ee}{\end{equation}}

\begin{document}

\title{Radial solutions for Hamiltonian Elliptic Systems with weights}

\author{Pablo L. De N\'apoli}
\address{Departamento de Matem\'atica, Facultad de Ciencias Exactas y Naturales,
Universidad de Buenos Aires, 1428 Buenos Aires, Argentina} \email{pdenapo@dm.uba.ar}

\author{Irene Drelichman}
\address{Departamento de Matem\'atica, Facultad de Ciencias Exactas y Naturales,
Universidad de Buenos Aires, 1428 Buenos Aires, Argentina} \email{irene@drelichman.com}

\author{Ricardo G. Dur\'an}
\address{Departamento de Matem\'atica, Facultad de Ciencias Exactas y Naturales,
Universidad de Buenos Aires, 1428 Buenos Aires, Argentina} \email{rduran@dm.uba.ar}


\thanks{Supported by ANPCyT under grant  PICT 2006-01307, by Universidad de Buenos Aires under grants X070 and X837 and by
CONICET under grants PIP 5477 and 5478. The authors are members of CONICET, Argentina.}
\begin{abstract}
We prove the existence of infinitely many radial solutions for elliptic systems
in $\R^n$ with power weights. A key tool for the proof will be a weighted imbedding theorem
for fractional-order Sobolev spaces, that could be of independent interest.
\end{abstract}

\maketitle

\section{Introduction}

In this paper, we study the existence of non-trivial, radially symmetric solutions of the
following Hamiltonian elliptic system in $\R^n$:
\be
 \left\{
\begin{array}{rclll}
- \Delta u + u & = &  |x|^a |v|^{p-2}v  \\
- \Delta v + v & = &  |x|^b |u|^{q-2}u  \\
\end{array}
\right.
\label{sistema-rn}
\ee
\medskip
More precisely, we will prove the following theorem:
\begin{thm}
Assume that the following conditions hold:
\be
p,q > 2, \; \frac{1}{p} + \frac{1}{q} < 1
\label{superlinear}
\ee
\be
0<a<\frac{(n-1)(p-2)}{2},\quad 0 < b < \frac{(n-1)(q-2)}{2}
\label{limite-pesos}
\ee
\be
\frac{n+a}{p} + \frac{n+b}{q} > n-2
\label{subcritical}
\ee
and
\be q < \frac{2(n+b)}{n-4}, \quad p < \frac{2(n+a)}{n-4} \quad
\hbox{if} \; n \geq 5
\label{condicion-extra}
\ee
Then (\ref{sistema-rn}) admits infinitely many radially symmetric
weak-solutions (see Definition \ref{weak-sol} below)
\label{main-result}
\end{thm}

\medskip

In order to explain the significance of our result, let us briefly review some of the related results
in the literature. The related scalar equation
\be \left\{
\begin{array}{rclll}
-\Delta u & = & |x|^a u^{p-1} & \hbox{in} & \Omega \\
u >0      &   &                    & \hbox{in} & \Omega \\
u = 0     &   &                    & \hbox{on}& \partial \Omega \\
\end{array} \right.
\label{scalar-equation}
\ee
with $\Omega$ being the unit ball in $\R^n$, $a>0$ and $p > 2$ is known in the
literature as the H\'enon equation, since M. H\'enon introduced it as a model of
spherically symmetric stellar clusters \cite{H}. It is well known that the
presence of the weight $|x|^a$ modifies the global homogeneity of the
equation, and shifts up the threshold between existence and non existence
given by the application of the Pohozaev Identity. Indeed,  W. M. Ni proved
in \cite{Ni} the existence of a solution of (\ref{scalar-equation})
if $2 < p <  2^* + \frac{2a}{n-2} $ (here $2^*=\frac{2n}{n-2}$ denotes the
usual Sobolev critical exponent). The solutions found by Ni are radial and
arise via application of the Mountain Pass Theorem in the space of radial
functions. Problems in the whole space have also been considered. For instance,
for the related equation
\be
-\Delta u + |x|^a u = |x|^b u^{p-1}, \quad u \in H^1(\R^n)
\label{scalar-whole-space}
\ee
P. Sintzoff \cite{S} proved the existence of infinitely many radial solutions in the case
\be n \geq 3, p>1, 2<p<  2^* + \frac{2a}{n-2} \; \hbox{and} \; 2b -
\left(1+\frac{p}{2}\right) a < (n-1)(p-2) \label{Sintzoff-result} \ee

\medskip

Elliptic systems like (\ref{sistema-rn}) have also been extensively studied. An important
feature of this system is its Hamiltonian structure, that allows us to find weak
solutions using the methods of critical point theory. More precisely,  its solutions are
given by the critical points of a functional in a suitable functional space (see
(\ref{our-functional}) below). We refer the reader to the excellent survey \cite{dF}
for more details on this subject.

As a model problem, we may consider the system
\be
\left\{
\begin{array}{rclll}
-\Delta u &=& |x|^a |v|^{p-2}v & \hbox{in} & \Omega \\
-\Delta v &=& |x|^b  |u|^{q-2}u & \hbox{in} & \Omega \\
\end{array}
\right.
\label{sistemita}
\ee
with Dirichlet boundary conditions ($u=v=0$ in $\partial \Omega$), where $\Omega \subset \R^n$
is a bounded domain. It is known that in the unweighted case
$a= b =0$ (\ref{sistemita})  admits infinitely many non-trivial
solutions if
$$ p,q > 2, \; \frac{1}{p} + \frac{1}{q} < 1  \quad (\hbox{superlinearity})
$$
\be \frac{1}{p} + \frac{1}{q} > 1 - \frac{2}{n} \quad (\hbox{critical
hyperbola})
\label{original-critical-hyperbola} \ee
$$ q < \frac{2n}{n-4}, \quad p < \frac{2n}{n-4} \quad
\hbox{if} \; n \geq 5$$
For this result, see the paper  \cite{BdF} by T. Bartsch, D. G. de Figueiredo and the previous works \cite{dFF} and
\cite{HV}. In that paper  the existence of infinitely many radial solutions for
the unweighted system in $\R^n$ is also proved (i.e. problem (\ref{sistema-rn}) with $a=b=0$).

\medskip

In view of the known results for the Hen\'on equation, one would expect that the presence of the weights
should also cause a shift in the critical hyperbola (\ref{original-critical-hyperbola}), and, indeed, in
\cite{dFPR}  D. G. de Figueiredo, I. Peral and J. Rossi extended these results to problem (\ref{sistemita})
with non-trivial weights in an  arbitrary bounded domain  $\Omega \subset \R^n$, with $0 \in \Omega$,
obtaining existence of many strong solutions and at least one positive solution under
conditions (\ref{superlinear}), (\ref{subcritical}) and (\ref{condicion-extra}). (Note that condition
(\ref{limite-pesos}) is not needed in the case of a bounded domain).

$$\frac{1}{p} + \frac{1}{q} < 1, \quad \frac{n+a}{p} + \frac{n+b}{q} > n-2
$$
and
$$ q < \frac{2(n+b)}{n-4}, \quad p < \frac{2(n+a)}{n-4} \quad
\hbox{if} \; n \geq 5$$

\medskip

Essentially, our aim is to complement the above results by proving existence of infinitely many radially symmetric solutions of the weighted system \ref{sistema-rn} in the whole of $\R^n$ under appropriate restrictions on the weights.

\medskip

A key role in our proof will be played by a weighted imbedding theorem for radial functions
(Theorem \ref{imbedding-theorem}) that we believe it could also be of independent
interest. The proof of this imbedding relies on some $L^p-L^q$ estimates for the so-called
Fourier-Bessel (or Hankel) transform proved by L. De Carli in \cite{DC}, and is
given in section \ref{seccion-imbedding}.

\medskip
It is worth noting that in the case of a bounded domain, the
corresponding imbedding theorem (Proposition 2.1 of \cite{dFPR}) can be obtained
by applying the classical imbedding theorem combined with H\"older's
inequality. However, for an unbounded domain this approach is not possible
since the power weights $|x|^r$ are not inegrable. On the other hand, our imbedding
theorem has some extra restrictions (\ref{cota-del-peso}) that do not appear in the bounded case
(and that is why we have conditions (\ref{limite-pesos}) in Theorem
\ref{main-result}).

\medskip
The compactness of our imbedding will also be of fundamental
importance (since it will allow us to prove a suitable form of the Palais-Smale
compactness condition, see lemma \ref{lemma-PS} below). The idea of obtaining
better imbedding properties (in particular, compactness) by restricting
to subespaces of radially symetric functions goes back to the works of W. Strauss \cite{St}
and W. Rother \cite{R}, and was futher generalized in different directions by
 P. L. Lions \cite{Lions} and W. Sickel and  L. Skrzypczak \cite{SiSk}.

\medskip

The other important ingredient of our proof of Theorem \ref{main-result}  is an
abstract minimax theorem  from \cite{BdF} that we quote in Section
\ref{abstract-theorem} in order to make this paper
self-contained. Finally, in Section \ref{prueba}, we complete the proof of
Theorem \ref{main-result} by checking all the conditions of that
theorem. Notice that we need to use this abstract minimax theorem instead of the much simpler used in the bounded domain case (Theorem 3.1 in \cite{dFPR}) since the proof of Lemma
3.2 in that paper (which is analogous to Lemma \ref{lema-geo2} in our paper)
essentially depends on the fact that the Laplacian has discrete
spectrum in a bounded domain (which is not the case for  the linear
part of (\ref{sistema-rn}) in $\R^n$).

\medskip

Finally, to complete the picture, let us observe that the same methods of \cite{dFPR}
could be used to obtain radially symmetric solutions of (\ref{sistemita}) if $\Omega$
is a ball, under  the same hypotheses of Theorem 1.1 in that paper (i.e. without
the restrictions (\ref{limite-pesos})), by working in a subspace of radially
symmetric functions  $H^s_{rad}(\R^n) \times H_{rad}^t(\R^n)$, and then using the principle
of Symmetric Criticality (see \cite{P}).

\medskip
For the sake of simplicity, we concentrate on this paper on the model problem (\ref{sistema-rn}), though it is
clear that the same techniques could be used to handle more general Hamiltonian elliptic systems of
the form:
\be
 \left\{
\begin{array}{rcl}
- \Delta u+u & = & H_v(|x|,u,v)  \\
- \Delta v+v & = & H_u(|x|,u,v)   \\
\end{array}
\right.
\label{sistema}
\ee
in $\R^n$ under suitable hypotheses on the Hamiltonian function $H$ (analogous to those in
\cite{BdF}).

\section{An imbedding theorem for fractional order Sobolev spaces with weights}

\label{seccion-imbedding}

Using an idea that goes back to \cite{dFF}, we shall work in the fractional
Sobolev spaces defined, as usual, by:
$$ H^s(\R^n)= \left \{u \in L^2(\R^n) : \int_{\R^n} (1+|\omega|^2)^s
|\hat{u}(\omega)|^2 \; d\omega  < +\infty \right  \}  $$
were $\hat{u}$ denotes the Fourier transform of $u$. This is a Hilbert space
with the inner product given by:
$$ \langle u, v \rangle = \frac{1}{(2\pi)^n} \int_{\R^n} (1+|\omega|^2)^s
\hat{u}(\omega)  \overline{\hat{v}(\omega)} \; d\omega $$

For the proof of Theorem \ref{main-result}, we shall need the following weighted imbedding  theorem for  $H_{rad}^s (\R^n)$, the subspace of radially symmetric functions of $H^s(\R^n)$\footnote{Notice that troughout this
section $p$ and $q$ are not the exponents in (\ref{sistema-rn}).}:

\begin{thm}
Let $0<s<\frac{n}{2}$, $2<q< 2^*_c = \frac{2(n+c)}{n-2s}$
then we have the compact imbedding
\be
H^s_{rad} (\mathbb{R}^n) \subset L^q(\mathbb{R}^n,|x|^c dx)
\label{our-imbedding}
\ee
provided that
\be
-2s < c < \frac{(n-1)(q-2)}{2}
\label{cota-del-peso}
\ee
\label{imbedding-theorem}
\end{thm}

\begin{rem}
The case $s=1$ was proved by W. Rother \cite{Rother} (this theorem was used
by P. Sintzoff in \cite{S} for problem (\ref{scalar-whole-space})). The case $c=0$ gives the classical Sobolev
imbedding (in the case of radial functions):
$$ H_{rad}^{s}(\R^n) \subset L^q(\R^n) \; \hbox{for} \; 2 <q < \frac{2n}{n-2s} $$
In that case, the compactness of the imbedding is given by the following
theorem of P. L. Lions \cite{Lions}:
\end{rem}

\begin{thm}
Assume that $0<s<n/2$ and $2<q< \frac{2n}{n-2s}$, then the imbeding
$$ H^{s}_{rad}(\R^n) \subset L^q(\R^n)  $$
is compact.
\end{thm}

The proof of Theorem (\ref{our-imbedding}) is based on an $L^p-L^q$ estimate for the
so-called Fourier-Bessel (or Hankel) transform due to L. De Carli \cite{DC}. In this section,
we shall therefore follow the notations in that paper, that we recall
here for sake of completeness:

For given parameters $\alpha, \nu, \mu$ De Carli introduced the operator
$$ L^\alpha_{\nu,\mu} f(y)= y^\mu  \int_0^\infty (xy)^\nu f(x)
J_\alpha(xy) \; dx $$
where $J_\alpha$ denotes the Bessel function of order $\alpha$. A particular case of this
operator is the Fourier-Bessel transform:
\be
\tilde{\mathcal{H}}_\alpha f(x) = L^\alpha_{\alpha+1,-2\alpha-1} f(x)
\label{Fourier-Bessel-definition}
\ee

The importance of this operator for our purposes is due to the fact that it provides
an expresion for the Fourier transform of a radial function $u(x)=u_0(|x|)$

\be
\label{radial-transform}
 \hat{u}(|\omega|) = (2\pi)^{n/2} \tilde{H}_{n/2-1} (u_0) (|\omega|)
 \ee

Morover, we recall that we have the inversion formula
\be
\tilde{\mathcal{H}}_\alpha (\tilde{\mathcal{H}}_\alpha (u))(x) = u(x)
\,
\hbox{ (equation (2.4) from \cite{DC}) }
\label{inversion-formula}
\ee

Now, we state De Carli's theorem (Theorem 1.1 in \cite{DC}):

\begin{thm}
$L^\alpha_{\nu,\mu}$ is a bounded operator from $L^p(0,\infty)$ to $L^q(0,\infty)$
whenever $\alpha \geq -\frac{1}{2}$, $1\le p\le q\le \infty$ if and only if

$$ \mu = \frac{1}{p^\prime} - \frac{1}{q} \; \hbox{ and } -\alpha - \frac{1}{p^\prime}
< \nu \leq \frac{1}{2}-max\left(\frac{1}{p^\prime} - \frac{1}{q}, 0\right) $$

\label{De-Carli}
\end{thm}

Finally, we observe that the De Carli operators $L^\alpha_{\nu,\mu}$ enjoy two
invariance properties that will be useful in obtaining weighted estimates (and
that are immediate from their definition):

\be
y^e L^\alpha_{\nu,\mu}(f)(y) = L^\alpha_{\nu,\mu+e}(f)(y)
\label{invariance1}
\ee

\be
L^\alpha_{\nu,\mu}( f) = L^\alpha_{\nu-\sigma,\mu+\sigma} (x^\sigma f)
\label{invariance2}
\ee

Now we are ready to give a proof of Theorem \ref{imbedding-theorem}:
\begin{proof}
Let $u(x)= u_0(|x|) \in H^s_{rad}(\R^n)$. Using polar coordinates we have that:

$$ \left( \int_{\R^n} |x|^c |u|^q \; dx \right)^{1/q} = c_n
 \left(  \int_0^\infty r^{c+n-1} |u_0(r)|^q \; dr \right)^{1/q} $$
Thanks to the inversion formula (\ref{inversion-formula}) for the  Fourier-Bessel transform of
order $\alpha=\frac{n}{2}-1$ (which is just the usual Fourier inversion
formula for radial functions) we obtain:
$$= c_n
 \left(  \int_0^\infty r^{c+n-1} |\tilde{\mathcal{H}}_\alpha ( \tilde{\mathcal{H}}_\alpha(u_0) )(r)|^q \; dr \right)^{1/q}
$$
$$
=  c_n
 \left(  \int_0^\infty r^{c+n-1} |L^\alpha_{\alpha+1,-2\alpha-1} ( \tilde{\mathcal{H}}_\alpha(u_0) )(r)|^q \; dr
\right)^{1/q} \; \hbox{(using (\ref{Fourier-Bessel-definition}))} $$

Using (\ref{invariance1}) this can be written as:
$$  =c_n
 \left(  \int_0^\infty | L^\alpha_{\alpha+1,-2\alpha-1+\frac{c+n-1}{q}} ( \tilde{\mathcal{H}}_\alpha(u_0) )(r)|^q \; dr
\right)^{1/q} $$
Applying (\ref{invariance2}):
$$ = c_n \left(  \int_0^\infty | L^\alpha_{\alpha+1-\sigma,-2\alpha-1+\frac{c+n-1}{q}+\sigma }
( r^\sigma \tilde{\mathcal{H}}_\alpha )(r)|^q \; dr \right)^{1/q} $$
where the value of the parameter $\sigma$ will be chosen later.

Now we apply Theorem \ref{De-Carli}, with the following choice of parametes
$$ p= \frac{nq}{nq-n-c}, \quad \sigma=\frac{n-1}{p}, \quad
\alpha=\frac{n}{2}-1,\quad \nu= \alpha+1-\sigma,\quad \mu= -2\alpha-1+\frac{c+n-1}{q}+\sigma $$
and we get the bound:

$$ \left( \int_{\R^n} |x|^c |u|^q \; dx \right)^{1/q}  \leq C \left(  \int_0^\infty |r^\sigma \tilde{\mathcal{H}}_\alpha(u_0) (r)|^p \; dr
\right)^{1/p} .$$
Since it easy to see that, under the hypotheses of our theorem, all the restrictions
of Theorem \ref{De-Carli} are fulfilled, this bound is equal to:
$$= C \left(  \int_0^\infty (1+r^2)^{sp/2} (1+r^2)^{-sp/2}
|\tilde{\mathcal{H}}_\alpha(u_0) (r)|^p \; r^{n-1} \; dr
\right)^{1/p} $$
and using H\"older's inequality with exponent $2/p$,
$$ \leq C \left(  \int_0^\infty (1+r^2)^{s}  |\tilde{\mathcal{H}}_\alpha(u_0) (r)|^2  r^{n-1}\; dr
\right)^{1/2} \left( \int_0^\infty (1+r^2)^{-\frac{sp}{2-p}}  r^{n-1} \; dr \right)^{1/2} $$
$$ \leq C \| u \|_{H^s} $$
by (\ref{radial-transform}) and since under the restrictions of our theorem we have that
$$ \int_0^\infty (1+r^2)^{-\frac{sp}{2-p}} \; r^{n-1} \; dr < +\infty
\; \left( \hbox{recall that} \; \frac{2n}{2s+n}<p  \; \right) $$

It remains to prove that the imbedding
$H^s_{rad} (\mathbb{R}^n) \subset L^q(\mathbb{R}^n,|x|^c dx)$ is compact.
It is enough to show that if $u_n \to 0 $ weakly in $H^s$, then $u_n \to 0$ strongly
in $L^q(\R^n,|x|^c \; dx)$.

Since $2<q<2^*_{c}=\frac{2(n+c)}{n-2s}$ by hypothesis, it is possible to choose $r$ and $\tilde{q}$ such that
$2<r<q<\tilde{q}<2^*_{c}$.We write $q = \theta r + (1-\theta) \tilde{q}$ with $\theta \in (0,1)$ and,
using H\"older's  inequality, we have that

\be
\int_{\R^n} |x|^{c} |u_n|^q \, dx
\le \left( \int_{\R^n} |u_n|^r \, dx \right)^{\theta}
\left( \int_{\R^n} |x|^{\tilde{c}} |u_n|^{\tilde{q}} \, dx \right)^{1-\theta}
\label{ineq-interpolacion}
\ee

where $\tilde{c}=\frac{c}{1-\theta}$. By choosing $r$ close enough to $2$ (hence
making $\theta$ small), we can fulfill the conditions
$$ \tilde{q} <\frac{2(n+\tilde{c})}{n-2s}, -2s < \tilde{c} < \frac{(n-1)(\tilde{q}-2)}{2} $$
Therefore, by the imbedding that we have already established:
$$ \left( \int_{\R^n} |x|^{\tilde{c}} |u_n|^{\tilde{q}} \, dx  \right)^{1/\tilde{q}}
\leq C \| u_n \|_{H^s} \leq C^\prime $$

Since the imbeding $H_{rad}^s(\R^n) \subset L^{r}(\R^n)$ is compact by Lions theorem, we
have that $u_n \to 0$ in $L^{r}(\R^n)$. From (\ref{ineq-interpolacion}) we conclude
that $u_n \to 0$ strongly in $L^q(\R^n,|x|^c\; dx)$, which shows that the imbedding in our
theorem is also compact.
 \end{proof}

\section{An abstract critical point theorem}

\label{abstract-theorem}

In order to prove Theorem \ref{main-result} we will use an abstract  critical point result from
\cite{BdF}. For the reader's convenience, we will try to keep the notation from that paper.
We start by recalling the specific form of the Palais-Smale-Cerami
compactness condition used in \cite{BdF}:

\begin{defn}
We consider a Hilbert space $E$ and a functional $\Phi\in C^1(E,\R)$. Given a sequence
$(X_n)_{n \in \N}$ of finite dimensional subspaces of $X$, with $X_n \subset X_{n+1}$ and
$\overline{\bigcup X_n}=E$, the functional $\Phi$ is said to satisfy condition
$(PS)_c^{\mathcal{F}}$ at level $c$ if every sequence $(z_j)_{j \in \N}$ with $z_j
\in X_{n_j}$, $n_j \to +\infty$ and such that
$$\Phi(z_j) \to c \; \hbox{and} \; (1+\| z_j \|) (\Phi|_{X_{n_j}})^\prime(z_j) \to 0$$
(a so-called $(PS)_c^{\mathcal{F}}$ sequence) has a subsequence which
converges to a critical point of $\Phi$.
\end{defn}

\begin{thm}[Fountain Theorem, Theorem 2.2 from \cite{BdF}]
Assume that the Hilbert space $E$ splits as a direct sum $E=E^+ \oplus E^-$, and that
$E_1^\pm \subset E_2^\pm \subset \ldots  \subset E^\pm_n \subset $
are strictly increasing sequences of finite dimensional subspaces such that
$\overline{\bigcup_{n=1}^\infty  E_n^\pm} = E^\pm$ and let $E_n= E_n^+ \oplus
E_n^- $

Furthermore, assume that the functional $\Phi$ satisfies the following assumptions:
\begin{enumerate}
 \item [$(\Phi_1)$] $\Phi \in C^1(E,\R)$ and satisfies $(PS)_c^{\mathcal{F}}$ with respect to
$\mathcal{F}= (E_n)_{n \in \N}$ and every $c>0$.
\item [$(\Phi_2^\prime)$] There exists a sequence $r_k>0 \; (k \in \N)$ such that
$$ b_k = \inf \{ \Phi(z) : z \in E^+,\; z \perp E_{k-1} \; \| z \| = r_k \} $$
satisfy $b_k \to +\infty$.
\item [$(\Phi_3^\prime)$] There exist a sequence of isomorphisms $T_k:E \to E\; (k \in \N)$
with $T_k(E_n)=E_n$ for all $k$ and $n$, and there exists a sequence $R_k>0\; (k\in \N)$
such that, for $z=z^+ + z^- \in E_k^+ \oplus E^-$ with $\max(\| z^+ \|, \| z^-\| )=R$ one
has
$$ \| T_k \| > r_k \; \hbox{and} \; \Phi(T_k z) < 0 $$
\item [$(\Phi_4^\prime)$]
$d_k = \sup \{ \Phi(T_k(z^+ +z^-)) :
z^+ \in E_k^+, \;  z^- \in E^-, \; \| z^+ \|,\| z^- \|  \leq R_k \} < +\infty$
\item [$(\Phi_5)$] $\Phi$ is even, i.e. $\Phi(-z) = \Phi(z)\; \forall \; z \in E$.
\end{enumerate}
Then $\Phi$ has an unbounded sequence of critical values.
\label{Fountain-theorem}
\end{thm}

In our application, we will also use Remark 2.2 from \cite{BdF}, that we
state here as a lemma for the sake of completeness:

\begin{lemma}
Let $E$ be a Hilbert space, and $E_1 \subset E_2 \subset E_3 \subset \ldots $ be a sequence of finite dimensional subspaces of $E$ such that $E= \overline{\bigcup_{n=1}^\infty E_n}$. Assume that we
have a compact imbedding $E \subset X$, where $X$ is a Banach space.

Let $\Phi \in C^1(E,R)$ be a functional of the form $\Phi= P - \Psi$
where
$$ P(z) \geq \alpha \| z \|_E^p \;  \forall \; z \in E $$
and
$$ |\Psi(x)|  \leq \beta (1+\| z \|_X^q) \; \forall \; z \in E $$
where $\alpha$, $\beta$ and $q>p$ are positive constants. Then, there
exist $r_k>0 \; (k \in \mathbb{N})$ such that
$$ b_k = \inf \{ \Phi(z) : z \in E^+,\; z \perp E_{k-1} \; \| z \| = r_k \} \to +\infty $$
i.e. condition $(\Phi_2^\prime)$ in theorem \ref{Fountain-theorem} holds.
\label{BdF-2.2}
\end{lemma}

\section{Proof of the main Theorem}

\label{prueba}

\subsection{The Functional Setting}

Using conditions (\ref{superlinear}), (\ref{subcritical}) and ({\ref{condicion-extra}) we may choose $s,t$ such that $0<s,t<\frac{n}{2}$, $s+t=2$ and
$$ 2 <  p <  \frac{2(n+a)}{n-2t}, 2 < q < \frac{2(n+b)}{n-2s} $$
>From Theorem \ref{imbedding-theorem} we then have the compact imbeddings
\be
H^s_{rad} (\mathbb{R}^n) \subset L^q(\mathbb{R}^n,|x|^b dx), \quad
H^t_{rad} (\mathbb{R}^n) \subset L^p(\mathbb{R}^n,|x|^a dx)
\label{imbeddings}
\ee
The natural functional associated to (\ref{sistema}) is given by:
\be
\Phi(u,v) = \int_{\R^n} A^s u \cdot  A^t v - \int_{\R^n} H(x,u,v)
\label{our-functional}
\ee
in the subspace $E = H_{rad}^s(\R^n) \times H_{rad}^t(\R^n) \subset H^s(\R^n) \times H^t(\R^n)$, with
the pseudo-differential operator $ A^s u = (-\Delta+I)^{s/2} $ given in terms of the Fourier transform:
$$
\widehat{A^s u}(\omega)= (1+|\omega|^2)^{s/2} \hat{u}(\omega),
$$
and where $H$ is the Hamiltonian:
$$H(x,u,v)= \frac{|x|^b |u|^q}{q} + \frac{|x|^a |v|^p}{p}$$

The imbeddings (\ref{imbeddings}) imply that $\Phi$ is well defined in $E$ and $\Phi
\in C^1(E,\R)$ (see the appendix of \cite{R}).

\begin{defn}
\label{weak-sol}
We say that $z=(u,v)\in H_{rad}^s\times H_{rad}^t$ is an $(s,t)$-weak solution of the system (\ref{sistema-rn}) if $z$ is a critical point of the functional (\ref{our-functional}).
\end{defn}

\begin{rem}
The functional $\Phi$ is not well-defined in $H^s(\R^n) \times H^t(\R^n)$ because the imbedding (\ref{imbeddings})  is not valid in general for non-radial functions. However, the functional is well-defined in $\tilde E=(H^s(\R^n) \cap L^q(\mathbb{R}^n,|x|^b dx))\times (H^t(\R^n) \cap L^p(\mathbb{R}^n,|x|^a dx))$. Moreover, since $\Phi$ is invariant with respect to radial symmetries, the critical points of $\Phi$ in $E$ are also critical points in $\tilde E$ thanks to the Symmetric Criticality Principle (see Theorem 5.4 of \cite{P}).
\end{rem}

Next, consider the bilinear form $B:E \times E \to \R$ given by:
$$ B[z,\eta] := \int (A^s u A^t \phi + A^s \psi A^t v ) \; \hbox{where} \;
z=(u,v), \eta=(\psi,\phi) $$
Asociated with $B$ we have the quadratic form
$$ Q(z) = \frac{1}{2} B(z,z) = \int A^s u A^t v $$
It is well-known that the operator $L:E \to E$ defined by $\langle Lz, \eta
\rangle = B[z,\eta]$ has exactly two eigenvalues $+1$ and $-1$ and that the
corresponding eigenspaces are given by

$$
E^+ = \{ (u, A^{-t} A^s u) : u \in E^s\}, \qquad E^- =\{ (u, -A^{-t}A^s u) : u \in E^s \}.
$$
Then, we have that
\be
\Phi(z)= \frac{1}{2} \langle Lz, z \rangle  - \Psi(z)
\label{form-of-Phi}
\ee
where:
$$ \Psi(z) = \int H(x,u,v) $$

We now define the sequence of finite dimensional subspaces that we need to apply
Theorem \ref{Fountain-theorem}. For this purpose, choose an orthonormal basis $\{e_j\}_{j\in \mathbb{N}}$  of $H^s_{rad}(\R^n)$. By density, we can choose $e_j \in \mathcal{S}(\R^n)$ (the Schwarz class).
Then $f_j = A^{-t}A^s e_j$ form an orthonormal basis of $H^t_{rad}(\R^n)$,
$f_j \in \mathcal{S}(\R^n)$, and we may define the following finite dimensional
subspaces
$$ E_n^s = \langle e_j : j=1...n\rangle \subset H^s_{rad}(\R^n) $$
$$ E_n^t = \langle f_j : j=1...n\rangle \subset H^t_{rad}(\R^n) $$
$$ E_n = E_n^s \oplus E_n^t$$.

\subsection{The Palais-Smale condition}

In what follows, we will prove that the functional $\Phi$ satisfies
 conditions $(\Phi_1), (\Phi_2') - (\Phi_4'), (\Phi_5)$ in
Theorem \ref{Fountain-theorem}. We begin by checking the compactness
condition $(PS)_c^{\mathcal{F}}$:

\begin{lemma}
Condition $(\Phi_1)$ holds.
\label{lemma-PS}
\end{lemma}

\begin{proof}
Using the imbedding in Theorem \ref{imbedding-theorem}, it
follows from standard arguments (see for example \cite{R}) that $\Phi$ is well defined, and morover
$\Phi \in C^1(E,\R)$.

It remains to show that $\Phi$ satisfies the $(PS)_c^{\mathcal{F}}$ condition.
Assume that we have a sequence $z_{j} \in E_{n_j}$ such that
$\Phi(z_j)\to c, (1+\| z_j\|) (\Phi|_{E_{n_j}})' (z_j) \to 0 $.

We observe that since the functional $\Phi$ has the form
(\ref{form-of-Phi}) where $L:E \to E$ is a linear Fredholm operator of index
zero and  $\nabla \Psi: E \to E$ is completely continuous (due to the compactness
of the imbeddings (\ref{imbeddings}) ), then by Remark
2.1 of \cite{BdF}, it is enough to prove that $z_j$ is bounded.

Since $(\Phi|_{E_{n_j}})^\prime (z_j) \to 0$, in particular we have that
\be
|  \Phi^\prime  (z_j) (w) | \le C  \| w \|_E \; \hbox{for all} \; w \in E_{n_j}
\label{cota-phi-prime}
\ee

If $z_j=(u_j, v_j)$, taking $w_j = \frac{pq}{p + q} \left( \frac{1}{p} u_j,
\frac{1}{q} v_j\right)$, we have that

$$
C (1+\| w_j\|_E) \ge \Phi (z_j) - \Phi' (z_j)(w_j)
$$
$$
= \int A^s u_j A^t v_j - \int H(x,u_j,v_j)
$$
$$
- \left[ \frac{q}{p+q} \int A^s u_j A^t v_j +
\frac{p}{p+q}\int A^s u_j A^t v_j\right.
$$
$$\left.- \frac{q}{p+q} \int
H_u(x,u_j,v_j) u_j - \frac{p}{p+q} \int H_v(x,u_j,v_j) v_j
\right]
$$
$$
= \left( \frac{pq}{p+q}-1 \right) \int H(x,u_j,v_j)
$$
Using (\ref{superlinear}) and Theorem \ref{imbedding-theorem} we obtain
\be
\int H(u_j, v_j, x) \, dx = \int \frac{|x|^b |u_j|^q}{q} +
\frac{|x|^a |v_j|^p}{p} \leq C (1+\| u_j\|_{H^s} + \| v_j \|_{H^t} ).
\label{cota1}
\ee

Now, considering $w = (\psi, 0), \psi\in E^s_{n_j} \subset H^s_{rad}(\R^n)$ in (\ref{cota-phi-prime})
$$
\left| Q'(z_j)(w) \right| = \left| \int A^s \psi A^t v_j \right| \le
\int |H_u (u_j, v_j, x)  \psi | \, dx + C\| \psi \|_{H^s}
$$
$$ = \int |x|^b |u_j|^{q-2} u_j \psi + C\| \psi \|_{H^s} $$
$$
\leq\|u_j \|_{L^q(|x|^{b})}^{q-1}
\|\psi\|_{L^q(|x|^b)}  + C\|\psi \|_{H^s}
$$
and using Theorem \ref{imbedding-theorem} we conclude that
$$
\left| \int A^s \psi A^t v_j \right| \le C
\left( \|u_j\|_{L^q(|x|^b)}^{q-1} + 1 \right) \|\psi \|_{H^s}
$$
Using a duality argument (and the fact that $\int A^s \psi A^t v_j = 0 \;
\forall \psi \in (E_{n_j}^s)^{\perp}$), this implies that
\be \| v_j \|_{H^t} \leq  C \left( \|u_j\|_{L^q(|x|^b)}^{q-1}  + 1\right)
\label{cota2} \ee

Similarly, taking $w = (0,\psi), \psi\in E^t_{n_j}$ in (\ref{cota-phi-prime}), we obtain
$$
| Q^\prime(z_j)(w)  | = \left| \int A^s u_j A^t \psi \right| \le
C \left( \| v\|_{L^p(|x|^a)}^{p-1}  + 1 \right) \|\psi \|_{H^t}
$$
 hence,
\be
\| u_j \|_{H^s} \leq C \left( \| v_j\|_{L^p(|x|^a)}^{p-1}  + 1 \right)
\label{cota3}
\ee

Therefore, replacing (\ref{cota2}) and (\ref{cota3}) into (\ref{cota1}),we obtain
$$ \left( \frac{1}{C} \| u_j \|_{H^s} -1 \right)^{q/(q-1)} +
\left( \frac{1}{C} \| v_j \|_{H^t} -1 \right)^{p/(p-1)}
\leq C \left( 1 + \|u_j \|_{H^s} + \| v_j \|_{H^t} \right) $$

Since $p,q>1$, we conclude that $z_j$ is bounded in $E$, as we have claimed. It follows that $\Phi$ satisfies the $(PS)_c^{\mathcal{F}}$ condition.
\end{proof}

\begin{lemma}
Condition $(\Phi_2\prime)$ holds.
\end{lemma}

\begin{proof}
We follow the proof of Lemma 3.2 of \cite{BdF}. We apply Lemma \ref{BdF-2.2} in $E^+$, with
$$
P(z):= Q(z) = \int A^s u A^t v, \quad \Psi(z) := \int H(x,u,v) \, dx
$$
Since $z\in E^+$,
$$
Q(z) = \int A^s u A^t (A^{-t} A^s u) \, dx = \int |A^s u|^2 \, dx = \|u\|_{H^s}^2 = \frac{1}{2} \|z\|_E^2
$$
Using the imbeddings (\ref{imbeddings}), we have that
$$
\left| \int H(x,u,v) \right| \le  C \left( \|u\|_{H^s}^q + \|v\|_{H^t}^p \right)
\le C \left( \|z\|_E^q + \|z\|_E^p \right) \le \|z\|_E^{\max(p,q)}
$$
Thus, we have $\Phi = P - \psi$  with $P(z) \ge \frac{1}{2} \|z\|_E^2$ and $|\psi(z)| = |\int H(x,u,v) |\le C (1 + \|z\|_E^{\max(p,q)})$, with $\max(p,q)>2$. Therefore, by Lemma \ref{BdF-2.2}, condition $(\Phi_2\prime)$ holds.
\end{proof}

\subsection{The Geometry of the Functional $\Phi$}

In the next two lemmas, we check the requiered conditions on the geometry of the
functional $\Phi$:

\begin{lemma}
Condition $(\Phi_3\prime)$ holds.
\label{lema-geo2}
\end{lemma}

\begin{proof}
We follow the proof of Lemma 5.1 of \cite{BdF}.
We want to prove that there exist isomorphisms $T_k: E \to E \; (k\in \mathbb{N})$ such that $T_k(E_n)= E_n$ for all $k, n$ and that there exist $R_k >0  (k\in \mathbb{N})$ such that, if $z = z^+ + z^- \in E_k^+ \oplus E_k^- $ with $R_k = \max (\|z^+\|, \|z^-\|)$, then $\|T_k z\| > r_k$ and $\phi(T_k z) <0$ ($r_k$ being the same as that in condition $(\Phi_2\prime)$).

We want to see that there exists $\lambda_k$ such that the above condition holds with  $T_k = T_{\lambda_k}$ and $R_k = \lambda_k$, where
$$
T_{\lambda_k}(u,v) = (\lambda_k^\mu u, \lambda_k^\nu v) \qquad \mbox{with }
\mu =\frac{m-q}{q}, \nu = \frac{m-p}{p}, \quad m> \max(p,q).
$$

Clearly, $T_k: E \to E$ is isomorphism for all $k$. Moreover, $T_{\lambda_k} E_n = E_n$ for all $k$ and, for all $\lambda > 0$, we have that
\be
\int_{\R^n} H(x,T_\lambda z) \ge C \left( \lambda^{\mu q} \int_{\R^n} |u|^q |x|^b +
\lambda^{\nu p} \int_{\R^n} |v|^p |x|^a \right)
\label{cota-inferior-H}
\ee

For $z = z^+ + z^- \in E_k^+ \oplus E_k^-,$ let $z^- = z_1^- + z_2^-$ with $z_1^- \in E_k^-$ and $z_2^- \perp E_k^-$, and let $\bar z = z^+ + z_1^-$. If $z= (u,v)$, we extend these definitions to $u$ and have that $\bar u = u^+ + u_1^- $  and, therefore, $u_2^- \perp \bar u$ in $L^2$.

Then,
$$
\| \bar u\|^2_{L^2} = |\langle \bar u, \bar u \rangle_{L^2}| =
 |\langle \bar u + u_2^- , \bar u \rangle_{L^2}| =
|\langle u, \bar u \rangle_{L^2}|
$$

$$
= \int_{\R^n} |u| |\bar u| |x|^{b/q} |x|^{-b/q} \le \| u \|_{L^q(|x|^b)}
\| \bar u\|_{L^{q'}(|x|^{-b/(q-1)})}
$$

But, since $\bar{u}\in E^k_s \subset \mathcal{S}(\R^n)$, we have that
$$
\|\bar u\|_{L^{q'}(|x|^{-b/(q-1)})} = \left( \int_{\R^n} |\bar u|^{q'}
|x|^{-b/(q-1)} \right)^{1/q'} < +\infty \; \hbox{since} \;  b< n(q-1) $$
and, thanks to the equivalence of the norms $\| \bar u\|_{L^2}$ and $\|\bar u\|_{L^{q'}(|x|^{-b/(q-1)})}$
(in the finite dimensional subespace $E^s_k$), we obtain
$$
\| u \|_{L^q(|x|^b)} \ge \gamma_k \|\bar u \|_{H^s} \; \forall \; u \in E^s_k
$$
for some $\gamma_k>0$. Similarly there exists $\overline{\gamma}_k >0$ such that
$$
\| v \|_{L^p(|x|^a)} \ge \overline{\gamma}_k \|\bar v \|_{H^t} \; \forall \; v \in E^t_k
$$

It then follows from (\ref{cota-inferior-H}) that
$$ \int_{\R^n} H(x,T_\lambda z) \geq C \left( \lambda^{\mu q} \; \gamma_k^q
\|\overline{u} \|_{H^s}^q + \lambda^{\nu p}  \; \overline{\gamma}_k^p
\| \overline{v} \|_{H^t}^p \right) $$

and (as in lemma 4.2 of \cite{BdF}) we get a lower bound of the form:
$$ \int_{\R^n} H(x,T_\lambda z) \geq c \min \left\{ \frac{1}{2^q}
\lambda^{\mu q} \gamma_k^{q} \lambda^q, \frac{1}{2^p} \lambda^{\nu p}
\overline{\gamma}_k^p \lambda^p \right\} \geq \sigma_k \lambda^{m} $$
provided that $\| z^+ \|_E=\lambda$.

On the other hand,
$$ Q(T_\lambda z) =  \lambda^{\nu+\mu} (\| z^+\|_E^2 - \| z^-\|_E^2) \leq
\lambda^{\nu+\mu+2} $$
for $\|z^+ \|_E= \lambda$. As a consequence, we have that
$$ \Phi(T_\lambda z) \leq \lambda^{\nu+\mu+2} - \sigma_k \lambda^m $$
Since $m > \nu + \mu +2 $, it follows that there is a $\lambda_0(k)$ such that
$T_{\lambda_k}(z)<0$ if $\lambda_k > \lambda_0(k)$.

Also we have that
$$ \| T_\lambda z \|_E \geq \lambda^{\min(\nu,\mu)} \| z \|_E^2 $$
which implies that
$$ \| T_\lambda z \|_E \geq \lambda_k^{\min(\mu,\nu)+2} \; \hbox{for}
\; \max(\| z^+ \|_E, \|z^- \|_E)= \lambda_k $$

Therefore, it is possible to select $\lambda_k>0$ such that
$$ \Phi(T_{\lambda_k} z) \leq 0 \; \hbox{and} \; \| T_{\lambda_k} z \|_E
\geq r_k $$
for any given $r_k$.
\end{proof}

Finally, we observe that condition $(\Phi_5)$ holds trivially. Therefore, all the conditions of Theorem
\ref{Fountain-theorem} are fulfilled, and hence the proof of Theorem
\ref{main-result} is complete.

\end{document}